\documentclass{elsarticle}

\usepackage{amssymb} 
\usepackage{amsmath} 
\usepackage{amsthm} 
\usepackage{units} 
\usepackage[OMLmathsfit]{isomath} 

\makeatletter
\def\ps@pprintTitle{%
  \let\@oddhead\@empty
  \let\@evenhead\@empty
  \def\@oddfoot{\reset@font\hfil\thepage\hfil}
  \let\@evenfoot\@oddfoot
}
\makeatother

\usepackage{fnbreak} 

\usepackage{enumitem}
\setitemize{itemsep=0pt}
\setenumerate{itemsep=0pt}

\usepackage[
  bookmarks=true,
  bookmarksnumbered=true,
  bookmarksopen=true,
  pdfborder={0 0 0},
]{hyperref}
\usepackage[all]{hypcap} 

\hypersetup{
  pdfauthor      = {Yannai A. Gonczarowski <yannai@gonch.name>},
  pdftitle       = {Satisfiability and Canonisation of Timely Constraints},
}

\theoremstyle{plain}
\newtheorem{thm}{Theorem}[section]

\newtheorem{lemma}[thm]{Lemma}
\newcommand{\lemmaref}[1]{Lemma~\ref{lemma:#1}}
\newtheorem{cor}[thm]{Corollary}
\newcommand{\cororef}[1]{Corollary~\ref{cor:#1}}
\newtheorem{observation}[thm]{Observation}
\newcommand{\observationref}[1]{Observation~\ref{observation:#1}}

\theoremstyle{definition}
\newtheorem{defn}[thm]{Definition}
\newcommand{\defnref}[1]{Definition~\ref{defn:#1}}
\newtheorem{remark}[thm]{Remark}

\newtheorem{ex}[thm]{Example}
\newcommand{\exref}[1]{Example~\ref{ex:#1}}

\newcommand{\sectionref}[1]{Section~\ref{section:#1}}

\newcommand{\tuple}[1]{\bar{#1}}
\newcommand{\tuplefull}[1]{\tuple{#1}=(#1_m)_{m=1}^n}

\newcommand{\eqdef}{\triangleq}

\newcommand{\distinctpairs}[1]{#1^{\bar{2}}}
\newcommand{\apath}{\tuple{p}}
\newcommand{\apathfull}{\tuplefull{p}}
\newcommand{\paths}[1]{\mathcal{P}(#1)}

\newcommand{\timeset}{\mathbb{T}}
\newcommand{\timediffgroup}{\mathbb{D}}

\newcommand{\implspec}{(I,\delta)}
\newcommand{\timpl}{\mathsans{t}}
\newcommand{\dpaths}{\paths{G_{\delta}}}

\newcommand{\dlength}{L_{G_{\delta}}}
\newcommand{\dtaglength}{L_{G_{\delta'}}}
\newcommand{\dtagtaglength}{L_{G_{\delta''}}}

\begin{document}
\begin{frontmatter}

\title{Satisfiability and Canonisation of \mbox{Timely Constraints}}
\author[yannai]{Yannai A. Gonczarowski}
\address[yannai]{Einstein Institute of Mathematics,
Hebrew University of Jerusalem, Israel
yannai@gonch.name}



\message{\abstractname}
\pdfbookmark[1]{\abstractname}{abstract}
\begin{abstract}
We abstractly formulate an analytic problem that arises naturally
in the study of coordination in multi-agent systems.
Let $I$ be a set of arbitrary cardinality (the set of \emph{actions})
and assume that for each pair of distinct actions $(i,j)$, we are given a number
$\delta(i,j)$.  We say that a function $\timpl$, specifying a time for each
action, \emph{satisfies} the \emph{timely constraint} $\delta$ if for every
pair of distinct actions
$(i,j)$, we have $\timpl(j)-\timpl(i) \le \delta(i,j)$ (and thus also $\timpl(j)-\timpl(i) \ge -\delta(j,i)$).
While the approach that first comes to mind for analysing these definitions
is an analytic/geometric one,
it turns out that graph-theoretic tools yield powerful results when applied
to these definitions.
Using such tools, we characterise the set of satisfiable
timely constraints, and reduce the problem of satisfiability of a timely
constraint to the all-pairs shortest-path problem, and for finite $I$,
furthermore to the negative-cycle detection problem.
Moreover, we constructively show
that every satisfiable timely constraint has a \emph{minimal} satisfying
function --- a key milestone on the way to optimally solving a large class of
coordination problems --- and reduce the problem of finding this minimal
satisfying function, as well as the problems of classifying and comparing
timely constraints, to the all-pairs shortest-path problem.
At the heart of our analysis lies the constructive definition of a ``nicely-behaved''
representative $\hat\delta$
for each class of timely constraints sharing the same set of satisfying
functions.
We show that this canonical representative, as well as the
map from such canonical representatives to the the sets of functions satisfying
the classes of timely constraints they represent, has many desired properties,
which provide deep insights into the structure underlying the above
definitions.
\end{abstract}


\end{frontmatter}


\section{Motivation and Definitions}
In a distributed algorithm, multiple processes, or agents, work toward
a common goal. More often than not, the actions of some agents are dependent
on the previous execution (if not also on the outcome) of the actions
of other agents. This, in turn, results in interdependencies between
the timings of the actions of the various agents. In this note,
we analyse such timing constraints in an abstract setting, and characterise
the satisfiability and the equivalence classes thereof. For a deeper look
into the motivation for the study in this note, the reader is referred to
\cite{halpern-moses-1990,fhmv-revisited,gomotark2013,bzm-wtr-ttr,mymsc}; for more information on the
application of the results described in this note,
the reader is referred to \cite{mymsc}.

\begin{defn}[Time]
Let $\timediffgroup \le \mathbb{R}$ be an additive subgroup of the real numbers that is
closed under
the infimum operation on bounded nonempty subsets.
We model \emph{time} as the nonnegative part of this group:
$\timeset \eqdef \{ t \in \timediffgroup \mid t \ge 0 \}$.
\end{defn}

\begin{remark}
The reader may verify that $\timediffgroup$ is either $\mathbb{R}$ (corresponding
to continuous modelling of time),
or cyclic (corresponding to discrete modelling of time) and hence isomorphic to
$\mathbb{Z}$.
In turn, $\timeset$ is either $\mathbb{R}_{\ge0}$ (the nonnegative real numbers),
or isomorphic to
$\mathbb{N}\cup\{0\}$,
respectively.
\end{remark}

\begin{defn}[Time Difference Bounds]
We define $\Delta = \timediffgroup \cup \{-\infty, \infty\}$.
\end{defn}

We now turn to model the timely constraints imposed on the actions of the
various agents.

\begin{defn}
Let $I$ be a set.
We denote the set of ordered pairs of distinct elements of $I$ by
$\distinctpairs{I} \eqdef \bigl\{(i,j) \in I^2 \mid i \ne j\bigr\}$.
\end{defn}

\begin{defn}[Satisfiability of Timely Constraints]\label{defn:satisfiability}
\leavevmode
\begin{enumerate}
\item
We call a pair $\implspec$ a \emph{timely specification}
if $I$ is a set (of arbitrary cardinality)
and if $\delta$ is a function
$\delta:\distinctpairs{I} \rightarrow \Delta$. We call $I$ the set of
\emph{actions} and call $\delta$ a \emph{timely constraint}.
\item
Let $\implspec$ be a timely specification. We say that a function
$\timpl : I \rightarrow \timeset$, specifying a time for every action,
\emph{satisfies}
$\delta$ (i.e.\ satisfies $(I,\delta)$),
if $\timpl$ satisfies $\timpl(j) \le \timpl(i) + \delta(i,j)$
for every $(i,j) \in \distinctpairs{I}$.
We denote the set of all functions satisfying $\delta$ by $T(\delta)$.
If $T(\delta) \ne \emptyset$, we say that $\delta$ is \emph{satisfiable};
otherwise, we say that $\delta$ is \emph{unsatisfiable}.
\end{enumerate}
\end{defn}

\begin{remark}
Obviously, $\delta$ is unsatisfiable unless $\delta>-\infty$
(in every coordinate).
Nonetheless, we still allow $\delta$ to take on the value of $-\infty$ for
some or all pairs of actions, for
technical reasons that become apparent when we define a canonisation operation
on timely constraints in \sectionref{satisfiability}.
\end{remark}

\begin{observation}\label{observation:satisfiability-properties}
Let $\implspec$ be a timely specification.
By \defnref{satisfiability},
\begin{itemize}
\item
$-\delta(j,i) \le \timpl(j) - \timpl(i) \le \delta(i,j)$, for every
$\timpl \in T(\delta)$ and every $(i,j) \in \distinctpairs{I}$.

\item
Let $\timpl:I \rightarrow \timeset$. If $\timpl \in T(\delta)$,
then $\timpl+c \in T(\delta)$ as well, for every $c \in \timeset$, as well as
for every other $c \in \timediffgroup$ s.t. $\timpl+c \ge 0$.

\item
$T$ is order-preserving: Let $\delta':\distinctpairs{I} \rightarrow \Delta$.
If $\delta' \le \delta$, then $T(\delta') \subseteq T(\delta)$.
\end{itemize}
\end{observation}

In the rest of this note, we embark on a graph-theoretic discussion
with the aim of analysing the above definitions.
In \sectionref{satisfiability},
we provide a more tangible characterisation for satisfiability of a timely
constraint,
define a ``nicely-behaved'' canonical representative $\hat{\delta}$ for
each set of timely constraints that share the same image under the mapping $T$,
and constructively build a minimal satisfying function for each such set.
In \sectionref{uniqueness}, we show that the quotient map
$\hat{\delta} \mapsto T(\delta)$,
for satisfiable timely constraints, is an order-embedding (i.e.\
it is both order-preserving and order-reflecting; thus it is also one-to-one),
and apply this result to further characterise the canonical form
$\hat{\delta}$.

\section{Satisfiability}\label{section:satisfiability}

As a first step toward analysing the satisfiability of a timely constraint,
we define a canonisation operation on timely constraints,
which preserves the set of satisfying functions.
In order to define the canonical form of a timely constraint $\delta$,
we consider $\delta$ as a weight function on the edges
of a directed graph on $I$.

\begin{defn}[Associated Graph]
Let $\implspec$ be a timely specification.
\begin{enumerate}
\item
We define the weighted directed graph of $\delta$ as
$G_{\delta} \eqdef (I,E_{\delta},\delta|_{E_{\delta}})$,
where the set of edges is defined as $E_{\delta} \eqdef \bigl\{(i,j) \in \distinctpairs{I} \mid \delta(i,j) < \infty\bigr\}$.
\item
We denote the set of paths in $G_{\delta}$ by $\dpaths$.
We denote the
length of a path $\apathfull \in \dpaths$ by
$\dlength(\apath) \eqdef \sum_{m=1}^{n-1} \delta(p_m, p_{m+1}) < \infty$.
\end{enumerate}
\end{defn}

\begin{defn}[Canonical Form]\label{defn:canonical-form}
Let $\implspec$ be a timely specification.
We define the \emph{canonical form} $\hat{\delta}$ of $\delta$ as
the distance function on $G_{\delta}$.
By slight abuse of notation,
we allow ourselves to write $\hat{\delta}$ instead of
$\hat{\delta}|_{\distinctpairs{I}}$ on some occasions below.
\end{defn}

\begin{observation}[Elementary Properties of the Canonical Form]\label{observation:canonical-form-properties}
Let $\implspec$ be a timely specification.
By \defnref{canonical-form}, we obtain the following properties
of $\hat{\delta}$:
\begin{itemize}

\item
$\forall i \in I: \hat{\delta}(i,i) \in \{0, -\infty\}$.
(Thus, by \observationref{satisfiability-properties}, for satisfiable
$\delta$ we obtain $\hat{\delta}|_{\{(i,i)\mid i \in I\}} \equiv 0$.)
Furthermore,
$\hat{\delta}(i,i)=-\infty$ iff $i$ is a vertex along a negative cycle in
$G_{\delta}$.

\item
Idempotence:
$\hat{\hat{\delta}}=\hat{\delta}$.

\item
Minimality\footnote{This name is justified in \sectionref{uniqueness}.}: $\hat{\delta} \le \delta$.

\item
Triangle inequality: $\forall i,j,k \in I:
\hat{\delta}(i,k) \le \hat{\delta}(i,j) + \hat{\delta}(j,k)$.

\item
Equivalence:
$T(\hat{\delta}) = T(\delta)$.
($\supseteq$:~by definition of $\hat{\delta}$.
$\subseteq$:~by minimality and by
\observationref{satisfiability-properties} (monotonicity of $T$).)

\item
Order preservation:
Let $\delta':\distinctpairs{I} \rightarrow \Delta$.
If $\delta' \le \delta$, then $\widehat{\delta'} \le \hat{\delta}$.

\end{itemize}
\end{observation}

We are now ready to characterise the satisfiable timely constraints on
a set $I$. The first part of the following lemma
performs this task, while its second part constructively shows that for every
satisfiable $\delta$, there exists a satisfying function that is minimal
in every coordinate --- a result that is of essence in order to
optimally solve a large class of
naturally-occurring coordination \mbox{problems~\cite{halpern-moses-1990,fhmv-revisited,gomotark2013,bzm-wtr-ttr,mymsc}}.

\begin{lemma}[Satisfiability Criterion]\label{lemma:satisfiable-iff}
Let $\implspec$ be a timely specification.
\begin{enumerate}
\item
$\delta$ is satisfiable
iff $\hat{\delta}|_{\{i\} \times I}$ is bounded from below for each $i \in I$.
\item
If $\delta$ is satisfiable, then
$i \mapsto -\inf(\hat{\delta}|_{\{i\} \times I})$
satisfies $\delta$, and is minimal in every coordinate with regard to this
property.\footnote{A quick glance at the formulation of this minimal satisfying function may raise
a suspicion that perhaps it would have been more natural to define $\delta$
as the negation (in every coordinate) of the definition we have given.
While it is indeed possible to define $\delta$ this way, and while doing so
would have indeed given a more natural definition of the minimal
satisfying function, it would have also required us to work with greatest path
lengths instead of distances, with a reverse triangle inequality and with
order-reversing monotonicity, which may seem less natural.}
\end{enumerate}
\end{lemma}

\begin{proof}
We first prove that if $\delta$ is satisfiable, then
for every $\timpl \in T(\delta)$ and for every $i \in I$, we have
$\timpl(i) \ge -\inf(\hat{\delta}|_{\{i\} \times I})$. This
implies the first direction (``$\Rightarrow$'') of the first part of the lemma,
and the minimality in the second part of the lemma.

Assume that $\delta$ is satisfiable and let $\timpl \in T(\delta)$.
By \observationref{canonical-form-properties} (equivalence),
$\timpl \in T(\hat{\delta})$ as well.
Let $i \in I$. By definition of satisfiability, we obtain
\[ \forall j \in I \setminus \{i\}: \hat{\delta}(i,j) \ge \timpl(j) - \timpl(i) \ge 0 - \timpl(i) = -\timpl(i). \]
By \observationref{canonical-form-properties}, $\hat{\delta}(i,i)=0\ge-\timpl(i)$.
Thus, we have $\hat{\delta}|_{\{i\}\times I} \ge -\timpl(i)$, and hence
$\inf(\hat{\delta}|_{\{i\}\times I}) \ge -\timpl(i)$, completing
this part of the proof.

We now prove that if $\hat\delta|_{\{i\} \times I}$ is bounded from below for
each $i \in I$, then the function defined in the second part of the lemma
indeed satisfies $\delta$. This completes the proof of both
parts of the lemma.

Define $\timpl:I\rightarrow\timeset$ by
$i \mapsto -\inf(\hat\delta|_{\{i\} \times I}) < \infty$.
For every $i \in I$, by \observationref{canonical-form-properties},
$\hat\delta(i,i)\le 0$, and therefore
$\timpl(i) \ge 0$. Thus, $\timpl$ is well defined. Let $(i,j) \in \distinctpairs{I}$ and
let $\apathfull \in \dpaths$ s.t.\ $p_1=j$. Define $p_0 \eqdef i$.
Note that
\[\inf(\hat\delta|_{\{i\}\times I}) \le \dlength\bigl((p_m)_{m=0}^n\bigr) =
\delta(i,j) + \dlength\bigl((p_m)_{m=1}^n\bigr).\]
By taking the infima of both sides over all $\apath \in \dpaths$ s.t.\
$p_1=j$, we obtain
$\inf(\hat\delta|_{\{i\} \times I}) \le
\delta(i,j) + \inf(\hat\delta|_{\{j\} \times I})$.
Thus, $\timpl(j) \le \timpl(i) + \delta(i,j)$, as required.
\end{proof}

\begin{ex}\label{ex:nonnegative-delta}
Let $(I,\delta)$ be a timely specification. If $\delta\ge 0$
(i.e.\ no lower bound is given on 
the proximity of any pair of actions), then
$\delta$ is satisfiable, and its minimal satisfying function is $\timpl\equiv 0$.
\end{ex}

For the case in which $I$ is finite, the first part of
\lemmaref{satisfiable-iff} yields the following, even more tangible, satisfiability
criterion.

\begin{cor}[Satisfiability Criterion --- Finite Case]\label{cor:satisfiable-iff-no-negative-cycles}
Let $\implspec$ be a timely specification s.t.\ $|I|<\infty$ and
$\delta>-\infty$.
$\delta$ is satisfiable iff $G_{\delta}$ contains no negative cycles.
\end{cor}

We conclude this section by showing, by means of a simple example,
that the finiteness condition in
\cororef{satisfiable-iff-no-negative-cycles} cannot be dropped.

\begin{ex}
Set $I\eqdef\mathbb{N}$. Define $\delta:\distinctpairs{I}\to\Delta$ by
$\delta(1,n) \eqdef -n$ for every \mbox{$n \in \mathbb{N}\setminus\{1\}$}, and $\infty$
in all other coordinates.
It is easy to see that $\delta$ is unimplementable
(either directly: what would $\timpl(1)$ be?; or using \lemmaref{satisfiable-iff},
as $\delta|_{\{1\} \times I}$ is unbounded from below, and therefore neither is
$\hat\delta|_{\{1\} \times I}$),
even though $G_{\delta}$ contains no negative cycles. (In fact,
$G_{\delta}$ is a star, and thus contains no cycles at all.)
\end{ex}

\section{Uniqueness of the Canonical Form}\label{section:uniqueness}

We now prove a uniqueness property
one may expect from the canonical form defined above, namely that
the equivalence classes of satisfiable timely constraints,
under the equivalence
relation $\delta_1 \sim \delta_2 \Leftrightarrow T(\delta_1)=T(\delta_2)$,
are in one-to-one correspondence with canonical forms. Furthermore, we show
that the quotient map $\nicefrac{T}{\sim}$, mapping canonical forms 
(as representatives of equivalence classes) to sets of
satisfying functions,
is an order-embedding. We use these results to deduce additional, equivalent,
definitions for the canonical form,
each shedding a different light thereon.
At the heart of all the results in this section lies the following lemma,
constructively demonstrating that each coordinate $\hat{\delta}(i,j)$ of the canonical form $\hat{\delta}$
of a timely constraint $\delta$ captures the upper bound imposed by $\delta$
on $\timpl(j)-\timpl(i)$ in the tightest manner possible.

\begin{lemma}[Attainability of Canonical Constraints]\label{lemma:canonical-form-bounds-attained}
Let $\implspec$ be a timely specification s.t.\ $\delta$ is satisfiable,
and let $\tilde{\imath},\tilde{\jmath} \in I$.
\begin{enumerate}
\item
If $\hat{\delta}(\tilde{\imath},\tilde{\jmath})<\infty$, then
there exists $\timpl \in T(\delta)$ satisfying
\mbox{$\timpl(\tilde{\jmath})-\timpl(\tilde{\imath})=\hat{\delta}(\tilde{\imath},\tilde{\jmath})$}.
\item
If $\hat{\delta}(\tilde{\imath},\tilde{\jmath})=\infty$,
then for every $K \in \timeset$,
there exists $\timpl \in T(\delta)$ satisfying
$\timpl(\tilde{\jmath})-\timpl(\tilde{\imath}) \ge K$.
\end{enumerate}
\end{lemma}

\begin{proof}
By \lemmaref{satisfiable-iff}, $\forall i \in I: \exists d_i \in \timeset:
\hat{\delta}|_{\{i\} \times I} \ge -d_i$. (For the time being, we may choose
$(-d_i)_{i \in I}$ to be the infima of the respective restrictions of
$\hat{\delta}$.) We define
$\delta':\distinctpairs{I} \rightarrow \Delta \setminus \{\infty\}$ by
\[ \forall i,j \in \distinctpairs{I}: \delta'(i,j) =
\begin{cases}
\delta(i,j) &\delta(i,j) < \infty \\
d_j         &\delta(i,j) = \infty.
\end{cases} \]
As $\delta' \le \delta$, by
\observationref{satisfiability-properties} (monotonicity of $T$)
it is enough
to find $\timpl \in T(\delta')$ that meets the conditions of the lemma.
By \observationref{canonical-form-properties} (equivalence),
this is equivalent to finding $\timpl \in T\bigl(\widehat{\delta'}\bigr)$ that meets
the conditions of the lemma.

We start by showing that $\forall i \in I: \widehat{\delta'}|_{\{i\} \times I} \ge -d_i$.
Let $i \in I$ and let $\apathfull \in \paths{G_{\delta'}}$ s.t.\ $p_1 = i$.
Set $l=\Bigl|\bigl\{k \in [n-1] \mid \delta(p_k,p_{k+1}) = \infty\bigr\}\Bigr|$,
where
$[n-1]\eqdef\{1,\ldots,n-1\}$; $l$ is
the number of ``new'' edges in $\apath$, which do not exist in $G_{\delta}$.
We show, by induction on $l$, that $\dtaglength(\apath) \ge -d_i$.

Base: If $l=0$, then $\dtaglength(\apath) = \dlength(\apath) \ge -d_i$.

Induction step: Assume that $l \ge 1$. Let $k \in [n-1]$ be maximal such that
$(p_k,p_{k+1})$ is a ``new'' edge (i.e.\ $\delta(p_k,p_{k+1})=\infty$).
By definition of $\delta'$, we have $\delta'(p_k,p_{k+1})=d_{p_{k+1}}$.
Thus, by the induction hypothesis and by definition of $d_{p_{k+1}}$, we obtain
\begin{align*}
\dtaglength(\apath) =&\:
\dtaglength\bigl((p_m)_{m=1}^k\bigr) \:&+&\: \delta'(p_k,p_{k+1})
\:&+&\: \dlength\bigl((p_m)_{m=k+1}^n\bigr) \:&\ge&\: \\
\ge&\: -d_i \:&+&\: d_{p_{k+1}} \:&-&\: d_{p_{k+1}} \:&=&\: -d_i,
\end{align*}
and the proof by induction is complete. In particular, we conclude that \linebreak
$\widehat{\delta'} > -\infty$. (Recall that by definition, also
$\widehat{\delta'} \le \delta' < \infty$.)

We claim that
$\timpl \eqdef d_{\tilde{\imath}} + \widehat{\delta'}(\tilde{\imath},\cdot) \ge 0$
satisfies
$\widehat{\delta'}$ (and hence also satisfies $\delta'$ and $\delta$).
Indeed, by
\observationref{canonical-form-properties} (triangle inequality), for every $(j,k) \in \distinctpairs{I}$ we have
\[\timpl(k) = d_{\tilde{\imath}} + \widehat{\delta'}(\tilde{\imath},k) \le
d_{\tilde{\imath}} + \widehat{\delta'}(\tilde{\imath},j) +
\widehat{\delta'}(j,k) = \timpl(j) +
\widehat{\delta'}(j,k).\]

If $\hat{\delta}(\tilde{\imath},\tilde{\jmath})<\infty$, we define
$K \eqdef \hat{\delta}(\tilde{\imath},\tilde{\jmath})$; otherwise,
let $K \in \timeset$ be arbitrarily large as in the second part of the lemma.
As $\delta'$ is satisfiable, by \observationref{canonical-form-properties} we
obtain $\widehat{\delta'}(\tilde{\imath},\tilde{\imath})=0$.
Therefore,
\[\timpl(\tilde{\jmath}) - \timpl(\tilde{\imath}) =
\bigl(d_{\tilde{\imath}} + \widehat{\delta'}(\tilde{\imath},\tilde{\jmath})\bigr) -
\bigl(d_{\tilde{\imath}} + \widehat{\delta'}(\tilde{\imath},\tilde{\imath})\bigr) =
\widehat{\delta'}(\tilde{\imath},\tilde{\jmath}).\]
Thus, if $\widehat{\delta'}(\tilde{\imath},\tilde{\jmath})
\ge K$, then the proof is complete.
(For the case in which \linebreak
$\hat{\delta}(\tilde{\imath},\tilde{\jmath})<\infty$, we obtain $\widehat{\delta'}(\tilde{\imath},\tilde{\jmath})\le K$
by \observationref{canonical-form-properties} (monotonicity),
since $\delta' \le \delta$.)

Otherwise,
set $d \eqdef K - \widehat{\delta'}(\tilde{\imath},\tilde{\jmath}) >0$,
and for every $i \in I$ define $d_i' \eqdef d_i + d > d_i$.
Therefore, $-d_i' < -d_i \le \hat{\delta}|_{\{i\} \times I}$ for every $i \in I$.
Denote by $\delta''$ the function constructed from $\delta$
in the same way in which $\delta'$ was constructed from it,
but using the lower bounds $(-d_i')_{i \in I}$ rather than
$(-d_i)_{i \in I}$.
As explained above, in order to prove that $d_{\tilde{\imath}}'+\widehat{\delta''}(\tilde{\imath},\cdot)$ satisfies the conditions of the lemma,
it is enough to show that
$\widehat{\delta''}(\tilde{\imath},\tilde{\jmath}) \ge K$.
Let $\apathfull \in \paths{G_{\delta''}}$ s.t.\ $p_1 = \tilde{\imath}$ and
$p_n = \tilde{\jmath}$.
If $\forall k \in [n-1]: \delta(p_k, p_{k+1}) < \infty$, then
$\dtagtaglength(\apath) = \dlength(\apath) \ge \hat{\delta}(\tilde{\imath},\tilde{\jmath}) \ge K$.
Otherwise,
\begin{align*}
&\:\dtagtaglength(\apath) = & \hskip-1pt\text{by definitions of $\delta'$ and $\delta''$} \\
=&\:
\dtaglength(\apath) + d\cdot\Bigl|\bigl\{k \in [n-1] \mid \delta(p_k, p_{k+1}) = \infty \bigr\}\Bigr| \ge & \text{as this set is non-empty} \\
\ge&\: \dtaglength(\apath) + d \ge & \text{by definition of $\widehat{\delta'}$} \\
\ge&\: \widehat{\delta'}(\tilde{\imath},\tilde{\jmath}) + d = & \text{by definition of $K$} \\
=&\:  K.
\end{align*}
Either way, the proof is complete.
\end{proof}

While there exist unsatisfiable timely constraints whose canonical forms differ 
(due to $G_{\delta}$ not necessarily being
strongly connected and to $I$ no necessarily being finite), we now conclude, using
\lemmaref{canonical-form-bounds-attained}, that for satisfiable
timely constraints, the map $\hat{\delta}\mapsto T(\delta)$, from the canonical
form of a satisfiable timely constraint $\delta$ to the set of functions satisfying
$\delta$ (this map is well defined by \observationref{canonical-form-properties} ---
equivalence), is an order-embedding (and thus, in
particular, also one-to-one). This gives way to the use of the canonical form
as an efficient tool for classifying and sorting timely constraints according
to their ``strictness''.

\begin{cor}[$\hat{\delta} \mapsto T(\delta)$ is an Order-Embedding]\label{cor:canonical-form-order}
Let $I$ be a set and let
$\delta_1,\delta_2:\distinctpairs{I} \rightarrow \Delta$ s.t.\
$\delta_1$ is satisfiable.
$\hat{\delta}_1 \le \hat{\delta}_2$ iff $T(\delta_1) \subseteq T(\delta_2)$.
\end{cor}

\begin{proof}
$\Rightarrow$:
Assume that $\hat{\delta}_1 \le \hat{\delta}_2$.
By \observationref{satisfiability-properties} (monotonicity of $T$)
and by
\observationref{canonical-form-properties} (equivalence), we have
$T(\delta_1) = T(\hat{\delta}_1) \subseteq T(\hat{\delta}_2) = T(\delta_2)$.

$\Leftarrow$:
Assume that $\hat{\delta}_1 \nleq \hat{\delta}_2$. Thus, there exist
$\tilde{\imath},\tilde{\jmath} \in I$ s.t.\
$\hat{\delta}_1(\tilde{\imath},\tilde{\jmath}) > \hat{\delta}_2(\tilde{\imath},\tilde{\jmath})$.
If $\hat{\delta}_1(\tilde{\imath},\tilde{\jmath})<\infty$, then
by \lemmaref{canonical-form-bounds-attained} there exists
$\timpl \in T(\delta_1)$ s.t.\ $\timpl(\tilde{\jmath})-\timpl(\tilde{\imath}) =
\hat{\delta}_1(\tilde{\imath},\tilde{\jmath}) > \hat{\delta}_2(\tilde{\imath},\tilde{\jmath})$,
and thus $\timpl \in T(\delta_1) \setminus T(\delta_2)$, and the proof is
complete.

If $\hat{\delta}_1(\tilde{\imath},\tilde{\jmath})=\infty$, then
$\hat{\delta}_2(\tilde{\imath},\tilde{\jmath})<\infty$ and thus there exists
$K \in \timeset$ s.t.\ \linebreak
$K > \hat{\delta}_2(\tilde{\imath},\tilde{\jmath})$.
Similarly to the proof of the
previous case, by \lemmaref{canonical-form-bounds-attained} there exists
$\timpl \in T(\delta_1)$ s.t.\ $\timpl(\tilde{\jmath})-\timpl(\tilde{\imath}) \ge K
> \hat{\delta}_2(\tilde{\imath},\tilde{\jmath})$.
Once again, we obtain that $\timpl \in T(\delta_1) \setminus T(\delta_2)$, and the
proof is complete.
\end{proof}

\begin{cor}[Uniqueness of the Canonical Form]\label{cor:canonical-form-uniqueness}
Let $I$ be a set and let
$\delta_1,\delta_2:\distinctpairs{I} \rightarrow \Delta$ s.t.\ at least
one of them is satisfiable.
$\hat{\delta}_1 = \hat{\delta}_2$ iff $T(\delta_1) = T(\delta_2)$.
\end{cor}

\begin{proof}
$\Rightarrow$: 
Assume that $\hat{\delta}_1=\hat{\delta}_2$. By \observationref{canonical-form-properties} (equivalence),
we have $T(\delta_1)=T(\hat{\delta}_1)=T(\hat{\delta}_2)=T(\delta_2)$.

$\Leftarrow$:
Assume that $T(\delta_1) = T(\delta_2)$. Thus, since at least one of
$\delta_1,\delta_2$ is satisfiable, they both are.
To complete the proof, we apply \cororef{canonical-form-order} to
$T(\delta_1) \subseteq T(\delta_2)$ and to $T(\delta_2) \subseteq T(\delta_1)$.
\end{proof}

The above discussion gives rise to two alternative definitions (or rather,
characterisations) of the
canonical form of satisfiable functions.
The first one justifies
the name of the minimality property from
\observationref{canonical-form-properties}, and stems from this property when
combined with \cororef{canonical-form-uniqueness}.
The second one, which explicitly defines the inverse of the order-embedding
\mbox{$\hat{\delta}\mapsto T(\delta)$}, stems directly from
\lemmaref{canonical-form-bounds-attained} and from the definition of
satisfiability.
These definitions, both non-constructive in nature
(in contrast with \defnref{canonical-form}), showcase once more
the fact that the canonical form indeed emerges naturally, and that choosing
it to represent equivalence classes of timely constraints is not merely
an artifact of its being possible to constructively define and efficient to calculate.

\begin{cor}[Characterizations of the Canonical Form]\label{cor:canonical-form-characterisations}
Let $\implspec$ be a timely specification s.t.\ $\delta$ is satisfiable.
\begin{enumerate}
\item
$\hat{\delta} = \min\bigl\{\delta' \in \Delta^{(\distinctpairs{I})} \mid T(\delta')=T(\delta) \bigr\}$.\footnote{In particular, there exists a function in this set that is minimal
in every coordinate, however this may be directly proven by means of a simpler
argument.}
\item
$\forall i,j \in I: \hat{\delta}(i,j) = \sup\bigl\{\timpl(j)-\timpl(i) \mid \timpl \in T(\delta) \bigr\}$.\footnote{If this set is bounded from above,
then by \lemmaref{canonical-form-bounds-attained}, it attains its supremum.}
\end{enumerate}
\end{cor}

In fact, satisfiability of $\delta$ is not required in
\cororef{canonical-form-characterisations} if $G_{\delta}$ is
strongly connected and if $I$ is finite. Indeed, under these conditions, if $\delta$ is
unsatisfiable, then $\hat{\delta} \equiv -\infty$, which coincides with
both parts of this corollary when applied to any
unsatisfiable $\delta$.\footnote{When $G_{\delta}$ is
not strongly connected, then the coordinates $i$ for which
$\inf(\hat\delta|_{\{i\} \times I})$ assumes
the value of $-\infty$ (if $|I|<\infty$, these are the coordinates
for which $\hat{\delta}(i,i)=-\infty$) indicate the
connected component(s) of $G_{\delta}$ in which the ``reason(s)'' for unsatisfiability (e.g.\
the negative cycle(s), if $|I|<\infty$) lie(s).}
This suggests modifying the definition of the canonical
form of any unsatisfiable timely constraint to be $\hat{\delta} \eqdef -\infty$, in which
case Corollaries \ref{cor:canonical-form-order}, \ref{cor:canonical-form-uniqueness}
and \ref{cor:canonical-form-characterisations} apply to unsatisfiable timely constraints as well. As aesthetically-appealing
as such a definition may be, however, we note that it renders the canonical
form useless as a tool
for checking the solvability of a timely constraint (as the question of solvability
must be answered in order to compute the canonical form under this definition).
Indeed, in this case checking the solvability of a timely constraint
(and thus also computing its canonical form)
still involves computing $\hat{\delta}$ as it is defined in
\defnref{canonical-form}, and then applying the satisfiability
criterion from \lemmaref{satisfiable-iff} (and only then,
 if the timely constraint
turns out to be unsatisfiable, amending its canonical form to equal $-\infty$
across all coordinates).

\section{Discussion and Further Reading}
The problems discussed in this note seem analytic in
nature, perhaps more conventionally approached via geometric tools.
Nonetheless, as we have seen, these problems give way to natural
analysis by the somewhat unexpected use of graph theory, not only providing a
gamut of powerful and insightful theoretical results (regardless of the cardinality of $I$), but
also, when $I$ is finite, reducing the problem of satisfiability
of a timely constraint to that of negative-cycle detection, and the problems of finding minimal satisfying
functions and of classifying and comparing timely constraints --- to
all-pairs shortest-path computation (where finding a single coordinate of a
minimal satisfying function is reduced to single-source shortest-path computation).
This allows us to harness the vast existing knowledge regarding these
computational graph problems in order to efficiently analyse timely constraints.
For example, for a single timely constraint, all these problems may be
jointly solved via a single run of the Floyd-Warshall
algorithm~\cite{floyd-warshall}, in $O(|I|^3)$ time. Recently-discovered
all-pairs shortest-path algorithms, such as Han's~\cite{han},
provide even better asymptotic complexity.
In the special cases in which $\delta \ge 0$\footnote{In this case,
$\delta$ is always satisfiable, and finding a minimal satisfying function is
trivial (see \exref{nonnegative-delta});
however, classifying $\delta$ and comparing it to other timely constraints may
still be of interest.}
and/or $G_{\delta}$ is sparse (i.e.\
$\delta(i,j)=\infty$ for many $(i,j) \in \distinctpairs{I}$), other well-known
algorithms may be used to even further improve the running time;
the interested reader is referred to the notes concluding \cite[Chapter~25]{cormen}.

It should be noted that under many distributed models,
the graph $G_{\delta}$, associated with a timely constraint
$\delta$, plays an even more pivotal role in the study of timely
coordination than seen in this note.
For example, its strongly-connected components
are instrumental in characterising the communication channels 
(between the agents corresponding to the various actions) required for solving
timely-coordinated response problems associated with $\delta$. (For example,
all the coordination problems
introduced in \cite{halpern-moses-1990,fhmv-revisited,gomotark2013,bzm-wtr-ttr,bzm-ojr-gor} may readily be reformulated using an appropriate $\delta$ and analysed
in this way.)
For the details, which are beyond the scope of this note, the reader is
referred to~\cite{mymsc}.

\section{Acknowledgements}
The author would like to thank
Gil Kalai and Yoram Moses, the advisors of his M.Sc.\ thesis~\cite{mymsc},
upon chapter 5 thereof this note is based.


\message{\refname}
\pdfbookmark[1]{\refname}{references}
\def\bibsection{\section*{\refname}}
\bibliographystyle{model1-num-names}
\bibliography{satisfiability-and-canonisation}

\end{document}